\documentclass[11pt]{amsart}
\usepackage{graphicx}
\usepackage{amssymb}
\usepackage{epstopdf}
\usepackage{tikz}
\usetikzlibrary{cd}

\title{Generalized polynomial functors}
\author[J. D. Axtell]{Jonathan D. Axtell}
 \address{Sungkyunkwan University, Suwon 16419, Republic of Korea}
\email{jaxtell@skku.edu}

\begin{document}
\maketitle

\def \A {\mathcal A}
\def \B {\mathcal B}
\def \C {\mathcal C}
\def \D {\mathcal D}
\def \E {\mathcal E}
\def \G {\mathcal G}
\def \M {\mathcal M}
\def \P {\mathcal{P}}
\def \Q {\mathcal{Q}}

\def \EE {\boldsymbol{\E}}
\def \EG {\boldsymbol{\E\G}}
\def \GG {\boldsymbol{\G}}
\def \MM {\boldsymbol{\M}}
\def \PP {\boldsymbol{\P}}
\def \QQ {\boldsymbol{\Q}}
\def \VV {\boldsymbol{\V}}

\def \PPol {\boldsymbol{\P\hspace{-1pt}ol}}

\def \U {\mathcal U}
\def \V {\mathcal V}
\def \W {\mathcal W}
\def \Si{\mathfrak S}

\def \tensor{\otimes}

\def \k {\Bbbk}
\def \M {\mathcal{M}}
\def \Hom {\mathrm{Hom}}

\newcommand{\Pol}{\mathsf{Pol}}

\newcommand{\I}{\mathrm{I}}
\newcommand{\II}{\mathrm{II}}

\renewcommand{\hom}{\mathrm{hom}}
\newcommand{\eEnd}{\mathrm{end}}

\newcommand{\N}{\mathbb{N}}
\newcommand{\Z}{\mathbb{Z}}
\newcommand{\End}{\mathrm{End}}

\newcommand{\abs}[1]{|#1|}

\newcommand{\equi}{\stackrel{\sim}{\longrightarrow}}

\newcommand{\isoto}{\xrightarrow{\sim}}
\newcommand{\onto}{\twoheadrightarrow}

\def \mod {\mathbf{mod}}

\def \eg {\mathsf{eg}}
\def \ev {\mathsf{ev}}

\def \e {\mathsf{e}}

\def \mat {\mathit{M}}

\def \hm {\hspace{-0.5mm}}

\newcommand{\proj}{\mathbf{proj}}
\newcommand{\rrep}{\mathbf{rep}}
\newcommand{\rep}{\mathrm{rep}}

\mathchardef\hyphen="2D
\renewcommand{\-}{\kern 1pt \hyphen}

\newcommand{\0}{\bar{0}}
\newcommand{\1}{\bar{1}}

\newcommand{\bs}{\boldsymbol}

\newtheorem{theorem}{Theorem}[section]
\newtheorem{lemma}[theorem]{Lemma}
\newtheorem{proposition}[theorem]{Proposition}
\newtheorem{corollary}[theorem]{Corollary}
\theoremstyle{definition}
\newtheorem{definition}[theorem]{Definition}
\newtheorem{remark}[theorem]{Remark}
\newtheorem{example}[theorem]{Example}
\newtheorem{examples}[theorem]{Examples}

\begin{abstract}
We define Schur categories, $\Gamma^d \mathcal C$,
associated to a $\k$-linear category $\C$, 
over a commutative ring $\k$.
The corresponding representation categories, 
$\mathbf{rep}\, \Gamma^d\mathcal C$, 
generalize categories of strict polynomial functors.
Given a $\k$-superalgebra $A$, 
we show that 
for certain categories $\mathcal{V} = \VV_{\hm A}$, $\EE_{\hm A}$ of $A$-supermodules, 
 there is a Morita equivalence between 
$\mathbf{rep}\, \Gamma^d\mathcal{V}$
and the category of supermodules over generalized 
Schur superalgebras of the form 
$S^A(m|n,d)$ and $S^A(n,d)$, respectively. 
We also describe a formulation of generalized 
Schur-Weyl duality from the viewpoint of the category 
$\mathbf{rep}\, \Gamma^d \EE_{\hm A}$.
\end{abstract}

\section{Introduction}
Let $\k$ be a commutative ring and suppose $A=A_{\0}\oplus A_{\1}$ is a  
$\k$-superalgebra 
such that each $A_{\epsilon}$ ($\epsilon =\0,\1$) is a finitely generated, free $\k$-module.
The generalized Schur superalgebras $S^A(n,d)$ were
defined by Evseev and Kleshchev  \cite{EK1, EK2} 
in order to prove the Turner double conjecture.
These superalgebras are related to 
wreath product algebras by a
generalized Schur-Weyl duality established in \cite{EK1}. 
%\smallskip

The category $\P_{d}$
of (degree $d$, homogeneous) strict polynomial functors
was defined by Friedlander and Suslin in \cite{FS} 
in order to prove results about the cohomology of finite group schemes.
In the case where $\k$ is a field of 
characteristic $p \neq 2$, 
certain super-analogues of these categories were defined in \cite{A1},
denoted 
$\mathsf{Pol}^{(\I)}_d$ and $\mathsf{Pol}^{(\II)}_d$,
which are isomorphic to 
the categories of finitely generated supermodules over
the Schur superalgebras,
$S(m|n,d)$ and $Q(n,d)$, respectively,
provided that $m,n\geq d$.

In this paper, we describe a generalization of the categories 
$\P_{d}$ and $\Pol_d^{(\dagger)}$ ($\dagger= \I, \II$).
Suppose first that $\C$ is a $\k$-linear category whose 
Hom sets belong to the category 
$\VV_\k$
of finitely generated, projective  $\k$-supermodules, 
and such that composition induces an even $\k$-linear map
(see Section \ref{ss-enriched} for more details).
In Section \ref{sec-divided}, we define 
the Schur categories, $\Gamma^d \C$, 
whose morphisms consist of divided powers of 
morphisms in the category $\C$.
Then in Section \ref{sec-gen-poly}, 
we define the category of generalized polynomial functors
$\PP_{d,\C}$ to be the representation category
\[\PP_{d,\C} := 
\rrep\, \Gamma^d\C \]
consisting of all even, $\k$-linear functors 
from $\Gamma^d\C$ to $\VV_\k$.
%
%\smallskip

If the Hom-sets of  $\C$ 
belong to the subcategory
$\V_\k \subset  \VV_\k$
of ordinary  
(non-graded) projective $\k$-modules,
then the morphisms of $\Gamma^d\C$
also belong to $\V_\k$.  
In this case, we also define the category of generalized 
polynomial functors $\P_{d,\C}$ to be the representation category
 \[ \P_{d,\C} := \rep\, \Gamma^d\C,\]
consisting of all $\k$-linear functors from $\Gamma^d \C$ to $\V_\k$. 
%\smallskip

%
Now suppose that $A\in \VV_\k$ is a $\k$-superalgebra as described above.
Let $\VV_A$ denote the category of all finitely-generated, projective right $A$-supermodules. 
We then introduce  
{\em evenly-generated} $A$-supermodules in Section \ref{ss-eg}, 
and we let
$\EE_A \subset \VV_A$
denote the full subcategory 
of evenly-generated, projective right 
$A$-supermodules.  

Suppose that $m,n\geq d$ are integers.
Then we also describe a generalized Schur superalgebra of the form
$S^A(m|n,d)$ in Section \ref{ss-gen-schur}.
In our main result Theorem \ref{thm-main-1}, we show that 
there are equivalences 
of categories
\[\PP_{d,\VV_A}\ \simeq\ S^A(m|n,d) \- \mathbf{mod},\qquad
\PP_{d,\EE_A}\ \simeq\ S^A(n,d) \- \mathbf{mod}, \]
between a category of generalized polynomial functors and a 
corresponding category of left supermodules over a
Schur superalgebra.
%\smallskip

Considering the special cases where $A= \k$ or $A=\C(1)$, a 
rank one Clifford algebra, gives the following equivalences of categories: 
\[\PP_{d,\VV_\k} \simeq \Pol^{(\I)}_d, \qquad
\PP_{d,\EE_{\C(1)}} \simeq \Pol^{(\II)}_d\]
respectively. 
Similarly, we have isomorphisms of superalgebras:
\[S^{C(1)}(n,d) \cong Q(n,d), \qquad
S^\k(m|n,d) \cong S(m|n,d).\]
It follows that Theorem \ref{thm-main-1} generalizes
Theorem 4.2 of \cite{A1}.

Now consider the case of an ordinary ungraded algebra $A$.
Then we may regard $S^A(n,d)$ as an ordinary algebra.
In Theorem \ref{thm-main-2},
we describe an equivalence 
\[\P_{d,\V_A}\ \simeq\ S^A(n,d) \- \mathrm{mod}\]
between $\P_{d,\V_A}$
and the category of left 
$S^A(n,d)$-modules.
In the case $A=\k$, we have an equivalence of categories
\[\P_{d,\V_\k} \simeq \P_{d,\k},\]
and there is algebra isomorphism  between
$S^\k(n,d)$ and the classical Schur algebra $S(n,d)$ defined in \cite{Green}.
It follows that Theorem \ref{thm-main-2} generalizes
Theorem 3.2 of \cite{FS}.

Finally, in Section \ref{sec-gen-sw}, 
we describe an exact functor from the category 
$\PP_{d,\EE_A}$ 
to the category of
finite dimensional left supermodules over the wreath product superalgebra $A\wr\Si_d$
(Theorem \ref{thm-gen-sw}).
This functor may be viewed as a categorical analogue of the generalized Schur-Weyl duality  described by Kleshchev and Etseev in \cite{EK2}.

\section{Preliminaries}
We assume throughout that $\k$ is a commutative ring with unit. 
Let $\mathbb{N}$ and $\mathbb{N}_0$
denote the positive and nonnegative integers, respectively.

\subsection{Definitions}
Let $\mathcal{M}_\k$ denote the category of all $\k$-modules and $\k$-linear maps.
Given 
$M,N \in \mathcal{M}_\k$, we write $M\otimes N := M\otimes_\k N$,
%:
 $\mathrm{Hom}(M,N) := \mathrm{Hom}_\k(M,N)$ 
 and
 $\End(M):= \mathrm{Hom}(M,M)$.
%:

A $\k$-{\em supermodule} is a 
 $\mathbb{Z}/2$-graded $\k$-module, 
$M=M_{\0}\oplus M_{\1}$.
We write $|M| \in \mathcal{M}_\k$ to denote the ordinary $\k$-module
obtained by forgetting the $\mathbb{Z}/2$-grading.
Given 
 $\varepsilon \in \mathbb{Z}/2$ and homogeneous $v\in M_\varepsilon$, 
we also write $\overline{v}=\varepsilon$. 
Elements of $M_{\0}$ (resp.~$M_{\1}$) are called 
{\em even} (resp.~{\em odd}).  

Let us write $\MM_\k$ to denote the  category 
of all $\k$-supermodules and $\k$-linear maps. 
Given objects $M,N\in \MM_\k$, 
both the tensor product $M\otimes N$ and the hom-set 
$\mathrm{Hom}(M,N)$ 
are again objects of $\MM_\k$,
with $\mathbb{Z}/2$-gradings 
defined by 
\[(M\otimes N)_\epsilon := M_{\0}\otimes N_\epsilon \oplus M_{\1}\otimes N_{\1+\epsilon}\]
and 
\[\mathrm{Hom}(M,N)_\epsilon :=
\mathrm{Hom}(M_{\0},N_\epsilon) \oplus \mathrm{Hom}(M_{\1}, N_{\1+\epsilon}),\]
respectively, for each $\epsilon \in \mathbb{Z}/2$.

We will consider $\mathcal{M}_\k$ as a subcategory of $\boldsymbol{\mathcal{M}}_\k$
by identifying an ordinary $\k$-module $M\in \mathcal{M}_\k$ as a $\k$-supermodule concentrated in degree $\0$.

\subsection{Superalgebras}
A {\em superalgebra} 
is a $\k$-supermodule $A= A_{\0}\oplus A_{\1}$ which is a unital associative $\k$-algebra, 
such that multiplication
 $m_A : A\otimes A \to A$
 is an even $\k$-linear map.
A superalgebra homomorphism $\vartheta: A\to B$ is an even $\k$-linear map that 
is an algebra homomorphism in the usual sense.

Given a pair of superalgebras 
$A,B$, the tensor product $A\otimes B$ 
is again a superalgebra in a natural way. 
The multiplication is defined by 
the usual {\em rule of signs} convention 
\begin{equation}
\label{eq-rule}
(a_1\otimes b_1)(a_2 \otimes b_2) = (-1)^{\overline{b_1}\, \overline{a_2}} (a_1 a_2) \otimes (b_1 b_2)
\end{equation}
for all
$a_1,a_2 \in A$, $b_1,b_2 \in B$,
and the unit is given by $1_{A\otimes B} := 1_A \otimes 1_B$.
Expressions of the form \eqref{eq-rule} are assumed to hold for homogeneous elements
and are then extended linearly to the general case.

\subsection{Supermodules}
Let $A$ be a superalgebra.  
A {\em left $A$-supermodule} is
a $\k$-supermodule $M\in \boldsymbol{\mathcal{M}}_\k$ 
which is a left $A$-module in the usual sense, such that 
$A_\epsilon M_{\epsilon'} \subseteq M_{\epsilon+\epsilon'}$ for $\epsilon,\epsilon' \in \mathbb{Z}/2$.
One may define {\em right $A$-supermodules} similarly. 
A {\em subsupermodule}
of a left (resp.~right) $A$-supermodule 
$M$ is an $A$-submodule $N \subset M$ such that
$N = (N\cap M_{\0})\oplus 
(N\cap M_{\1})$.

A {\em homomorphism} $\varphi: M \rightarrow N$ 
of left (resp.~right) $A$-supermodules $M,N$ is a (not necessarily homogeneous) $\k$-linear map 
such that
\[\varphi(av) =(-1)^{\overline{\varphi}\, \overline{a}} a \varphi(v), 
\quad
\text{resp. }\  \varphi(v a) = \varphi(v)a,\]
for all $a \in A, v \in V$.

Considering $A$ as a right supermodule over itself, 
notice that the map 
\begin{equation}\label{eq-iso1}
A\, \xrightarrow{\sim}\, \mathrm{End}_A(A) :\, a\mapsto m_A(a\otimes-)
\end{equation}
is an isomorphism of superalgebras. 

Let ${}_A \MM$ (resp.~$\MM_A$) 
denote the category of all left (right) 
$A$-supermodules and homomorphisms, 
with composition of morphisms defined by the rule of signs, 
as in \eqref{eq-rule}. 
Given a pair $M,N$ of left (right) $A$-supermodules, 
let $\mathrm{Hom}_A(M,N)$ denote 
the set of homomorphisms, $\varphi: M\to N$.
We may also consider the superalgebra 
 $\mathrm{End}_A(M):=\mathrm{Hom}_A(M,M)$.

We write $A\-\mathbf{mod}$ (resp.~$\mathbf{mod}\-A$) 
to denote the full subcategory of
$A$-supermodules which are finitely generated as ordinary $A$-modules. 
The underlying $\k$-module $|A|\in \M_\k$ is an ordinary $\k$-algebra. 
In this case, we respectively write
\[{}_{|A|}\M, \quad  \M_{|A|}, \quad  |A|\-\mathrm{mod}, \quad \mathrm{mod}\-|A|,\]
to denote the corresponding categories of ordinary left $|A|$-modules, etc.

\subsection{Free $A$-supermodules} 
Let $A$ be a superalgebra.
Given any right $A$-supermodule $M\in \MM_A$,
define the {\em parity change} $\Pi M$ to be the same $A$-module, but with opposite 
$\mathbb{Z}/2$-grading.
For example, given $m,n\in \N$ and
considering $A$ as a right $A$-supermodule,
we write $A^{n|m} := A^n\oplus (\Pi A)^m$.
We will also write 
\[\mat_{n}(A) := \End_A(A^{n}),
\quad\qquad \mat_{n|m}(A):= \End_A(A^{n|m})\]
to denote corresponding superalgebras.

We define the (right) parity change functor $\Pi: \mod\-A  \to \mod\-A$ which sends $M \to \Pi M$.
On a right A-supermodule homomorphism $\phi : M \to N$, we set $\Pi(\phi) = (-1)^{\bar{\phi}}\phi$ as a linear map.

Suppose the free right $A$-supermodule $A^{n|m}$ (resp.~$A^{n'|m'}$) 
has homogeneous $A$-linear basis
$v_1, \dots, v_{m+n}$ (resp.~$w_1, \dots, w_{m'+n'}$),
such that 
$v_1, \dots, v_n$ (resp.~$w_1, \dots, w_{n'}$) are even
and 
$v_{n+1}, \dots, v_{m+n}$ (resp.~$w_{n'+1}, \dots, w_{m'+n'}$) are odd.
Then for a given $a\in A$, we write 
$e(j,i)$ (resp.~$e^{(a)}(j,i)$)
to denote the homomorphism
 in $\Hom_A(A^{n|m},A^{n'|m'})$ 
which sends $v_{k}\mapsto 0$ for all 
$k\neq i$ and sends
$v_i\mapsto w_j$ 
(resp.~$v_i\mapsto w_j a$).

Notice that the parity of the element $e(j,i)$ is given by:
\[\overline{e(j,i)} = 
\begin{cases}
\bar{0},& \text{if }\, 
1\leq i\leq n\, \text{ and }\, 1\leq j \leq n',\\
&\text{or if }\,
n< i\leq m+n\, \text{ and }\, n'< j \leq m'+n';\\
\bar{1}, &\text{else.}
\end{cases}\]
Similarly, given a homogeneous element $a\in A$, we have 
\begin{equation}\label{eq-parity}
\overline{e^{(a)}(j,i)} = 
\bar{a} + \overline{e(j,i)}.
\end{equation}

Now suppose $A\in \VV_\k$.  
If $A$ is free as a $\k$-supermodule, 
then we may identify
 $\Hom(\k^{n|m},\k^{n'|m'})$ as a
 $\k$-subsupermodule
of $\Hom_A(A^{n|m},A^{n'|m'})$
which is generated by the elements $e(j,i)$.  
The isomorphism \eqref{eq-iso1} may then be generalized as follows.

\begin{lemma}\label{lem:free}
Let $A$ be a superalgebra 
which is finitely-generated and free as a $\k$-module. 
Considering $A$ as a right $A$-supermodule, 
there exists for any $m,n,m',n'\in \N$
an even isomorphism of $\k$-supermodules
\[A\otimes \Hom(\k^{n|m},\k^{n'|m'}) 
\simeq
\Hom_A(A^{n|m},A^{n'|m'}): \quad
a\otimes e(j,i) \mapsto e^{(a)}(j,i).\]
In particular, we have the following superalgebra isomorphisms:
\\[-2.5mm]
\begin{itemize}
\item[(i)]
$A\otimes \End(\k^{n|m}) \simeq\mat_{n|m}(A)$ 
\\
\item[(ii)]
$A\otimes \End(\k^{n}) \simeq \mat_{n}(A)$
\end{itemize}
\end{lemma}

\begin{proof}
Since $A$ is a free $\k$-supermodule, we have 
$A\simeq \k^{p|q}$, for some $p,q \in \N_0$. 
Let $a_1, \dots, a_p$ 
(resp.~$a_{p+1}, \dots, a_{p+q}$)
be a $\k$-basis of $A_{\0}$ (resp.~$A_{\1}$). 
Then the set
\begin{equation*}\label{eq-basis}
\{e^{(r)}(j,i) := e^{(a_r)}(j,i)
\mid
1\leq i\leq m+n,\,
1\leq j\leq m'+n' ,\,
1\leq r \leq p+q\}
\end{equation*}
is a homogeneous $\k$-basis of $\Hom_A(A^{n|m},A^{n'|m'})$. 
It follows by 
\eqref{eq-parity} 
 that $\overline{e^{(r)}(j,i)}
= \overline{a}_r + \overline{e(j,i)}$, for any $r$.
The mapping 
\[a\otimes e(j,i) \mapsto e^{(a)}(j,i)\] 
thus yields an even isomorphism.
For the case $m=m'$ and $n=n'$,  
it is easy to check that 
the isomorphism 
\[A\otimes \End(\k^{n|m}) \simeq 
\End_A(A^{n|m})\]
is a superalgebra homomorphism. This shows part (i) of the lemma.  
Part (ii) follows by setting $m=0$.
\end{proof}

\subsection{Evenly generated supermodules}\label{ss-eg}
Notice that $\MM_A$ and $_A \MM$ are not abelian categories.  It will thus be convenient to consider the abelian subcategory 
 $_A\MM^{\ev}
 \subset {}_A \MM$ 
(resp.~
 $\MM_A^{\ev} 
 \subset \MM_A$), 
consisting of the same objects but only even morphisms.
This allows us to make use of the basic notions of homological algebra by restricting our attention to only even morphisms. 
For example, by a short exact sequence in $_A \MM$
 (resp.~$\MM_A$), we mean a sequence
\[0 \to L \to M \to N \to 0\]
with all the maps being even. 

Now suppose that $A$ is a superalgebra and 
that $M$ belongs to $A\-\mod$ (resp.\ $\mod\-A$). 
Then  $M$ is automatically finitely generated 
as an $A$-module. 
It follows that 
there is a short exact sequence of the form 
\[A^{n|m} \to M \to 0.\]
We will say that $M$ is {\em evenly generated} 
if there exists an even surjective homomorphism 
of the form
\[A^n \twoheadrightarrow M.\]

\subsection{Projective supermodules}
We say that a right $A$-supermodule $M$ 
is {\em projective}
if $M$ is a projective object of 
$\MM_A^\ev$. 
Let us write $\VV_A$ to denote the full subcategory 
of $\mod\-A$
consisting of all finitely generated, projective right $A$-supermodules.
Let us also write 
\[\EE_A \subset \VV_A\]
to denote the full subcategory consisting of the evenly-generated, projective right 
$A$-supermodules.  

\begin{remark}\label{rem-clifford}
The categories $\EE_A$ and $\VV_A$ are not distinct in general.
For example, let  $C(1)$ denote the rank 1 Clifford algebra 
(see \cite[Example 2.6]{A1}).
Then one may check that every $C(1)$-supermodule is evenly generated, 
and thus $\EE_{C(1)} = \VV_{C(1)}$.
\end{remark}

We further write $\V_{|A|}$ to denote the category of
ordinary (non-graded)
 finitely generated projective  $|A|$-modules.
Note for example that $\VV_\k$ (resp.~$\V_\k$) 
denotes the category 
of all finitely generated, projective $\k$-supermodules
(resp.~$\k$-modules).
In this case, we may identify $\EE_\k$ with $\V_\k$, considered as 
a subcategory of $\MM_\k$.

\subsection{Enriched categories}\label{ss-enriched}
We recall the definition of a specific type
of $\M$-enriched category,
where we assume that $\M \subset \MM_\k^{\ev}$.
See \cite{Kelly} for more details about 
$\M$-enriched categories in general.

First notice that 
$\MM_\k = (\MM_\k, \otimes, \k)$ 
is a symmetric monoidal category
with symmetry isomorphism 
given by the so-called {\em supertwist} map
\begin{equation*}
\tau: M\otimes N \xrightarrow{\sim} N\otimes M.
\end{equation*}
This is the  $\k$-linear map sending 
\begin{equation}\label{eq-twist}
v\otimes w \mapsto (-1)^{\overline{v}\, \overline{w}} w\otimes v \quad (v\in M, w\in N)
\end{equation}
 for any $M,N\in \MM_\k$.
It follows by \eqref{eq-twist} that the supertwist $\tau$ is even. 
It is thus natural to view 
 each of the subcategories 
\[\M_\k, \ \MM_\k^{\ev}\,
\subset\ \MM_\k\]
as a monoidal subcategory by restriction.

In the remainder, we assume 
 that $\M$ denotes a full subcategory of 
$\MM_\k^{\ev}$ such that: 
\begin{equation}\label{monoid}
\k \in \M\ \text{ and }\ 
M\otimes N \in \M,\,  \text{ whenever } M,N \in \M.
\end{equation}
It follows that $\M$ has the structure of a symmetric monoidal category
given by restriction.

We now recall the corresponding definition of enriched categories.

\begin{definition}[\cite{Kelly}]
An {\em $\M$-enriched category} (or {\em $\M$-category}) 
is a category $\C$ such that 
the hom-set $\Hom_\C(X,Y)$  
is an object of $\M$ for all $X,Y\in \C$, 
the composition is $\k$-bilinear and the induced map 
\[\circ^{\C}_{X,Y,Z}: \,
\Hom_{\C}(Y,Z) \otimes \Hom_{\C}(X,Y)
\ \rightarrow\ 
\Hom_{\C}(X,Z)\]
is an even $\k$-linear map (that is, a morphism in $\M$)
for all $X,Y,Z \in \C$. 
\end{definition}

If $\M \subset \M'$ are  both full subcategories of $\MM_{\k}^\ev$ 
satisfying \eqref{monoid}, 
then it is clear that any $\M$-category is automatically an $\M'$-category.
Notice for example that an $\M_{\k}$-enriched category is just a $\k$-linear category.

Let us fix a category $\M \subset \MM_\k^\ev$ as above.
We next recall the definition of enriched functors and natural transformations. 

\begin{definition}
Suppose $\C,\D$ are $\M$-enriched categories.  
An {\em $\M$-enriched} functor $F:\C\to \D$  
is a functor such that the map 
\[F_{X,Y}: \Hom_\C(X,Y) \to \Hom_{\D}(F(X),F(Y))\] 
is an even $\k$-linear map for any
pair of objects $X,Y \in \C$.  
\end{definition}

\begin{definition}
Next suppose that $F,G: \C \to \D$ is a pair of $\M$-enriched functors.  
We first define 
an {\em even} (resp.~{\em odd}) 
{\em $\M$-enriched natural transformation},
$\eta: F \to G$, 
to be a collection of even (odd) linear maps 
\[\eta(X) \in \Hom_{\M}(F(X),G(X)) \qquad (^\forall X\in \C)\]
such that for a given $\varphi \in \Hom_\M(X,Y)$ we have 
\[G(\varphi) \circ \eta(X) = 
(-1)^{\overline{\varphi}\, \overline{\eta(X)} } \eta(Y) \circ F(\varphi).\]
In general, an {\em $\M$-enriched natural transformation}, 
$\eta:F\to G$, is then defined 
to be a collection of linear maps, $\eta(X)= \eta_{\0}(X) \oplus \eta_{\1}(X)$ 
$(X\in \C),$ such that 
\[ \eta_\epsilon(X) \in \Hom_{\M}(F(X),G(X))_\epsilon \qquad (\epsilon \in \Z/2)\]
and $\eta_{\0}$ (resp.~$\eta_{\1}$): $F\to G$ is an even (resp.~odd)  $\M$-enriched
natural transformation.
\end{definition}

We also recall the corresponding category of enriched functors and natural transformations.

\begin{definition}
Given a pair $\C,\D$ of $\M$-enriched categories,
we write 
$\mathrm{Fun}_\k(\C,\D)$ to denote
the $\M$-enriched category consisting 
of all even $\k$-linear functors 
\[X: \C \to \D\]
and whose morphisms are $\M$-enriched natural transformations.
\end{definition}

Let $\C$ be an $\M$-enriched category.
We write $X \cong Y$,
if $X,Y$ are isomorphic in $\C$.  
If there is an even isomorphism $\varphi:X \cong Y$ (i.e., 
$\varphi\in \Hom_\C(X,Y)_{\bar{0}}$), then we use the notation 
$X \simeq Y$.

Let $\C^\mathrm{ev}$ 
denote the subcategory of $\C$ consisting of the same 
objects but only even morphisms.
Since $\M$-enriched 
functors send even morphisms to even morphisms, they
give rise to the corresponding functors between the underlying even 
subcategories.

Notice that $\V_\k$ and $\VV_\k^\ev$ are both full subcategories 
of $\MM_\k^\ev$ which satisfy \eqref{monoid}.
In the remainder, we will mostly consider enriched categories
where $\M=\V_\k$ or 
$\VV_\k^\ev$.

Given a superalgebra $A$, we may  view
$A\-\mod$ and $\mod\-A$ 
as $\VV_\k^{\ev}$-enriched categories,
while ${}_A\MM$ and $\MM_A$ 
are $\MM_\k^{\ev}$-enriched categories.
Furthermore, the categories 
${}_A\MM^\ev$ and $\MM_A^\ev$
are abelian, while the respective subcategories
$(A\- \mod)^\ev$, $(\mod\-A)^\ev$,
${}_A \VV^\ev$, and $\VV_A^\ev$
are exact categories in the sense of Quillen (see \cite{Buhler, Keller}).

\subsection{Tensor Products of $\M$-categories}
Suppose $\C, \D$ are $\M$-enriched categories. 
Then define $\C\otimes \D$ to be the
$\M$-enriched category whose objects are pairs
$(X,X')$ with $X\in \C$, $X'\in \D$
and with morphisms  defined by:
\[\Hom_{\C\otimes\D}((X,X'),(Y,Y'))
:=\Hom_{\C}(X,Y) \otimes\Hom_{\D}(X',Y')\]
for all $X,Y \in \C$ and $X',Y'\in \D$.
The composition of morphisms is defined similarly to \eqref{eq-rule}.

\section{Categories of Divided Powers}\label{sec-divided}
After recalling some basic properties of divided powers, 
we define the categories $\Gamma^d\C$ associated to 
any $\VV_\k^\ev$-category $\C$.

\subsection{Divided powers}
Suppose $M\in \VV_\k$ and $d\in \Z_{\geq 1}$. 
There is a unique (even) right action of the symmetric 
group $\Si_d$ on the tensor power $M^{\otimes d}$ such that each 
transposition $(i\ i+1)$ for $1\le i \le d-1$ acts by the supertwist: 
\[ (v_1 \otimes \cdots \otimes v_d). (i\ i+1) \ := 
\ (-1)^{\overline{v}_i \overline{v}_{i+1} } v_1\otimes \cdots \otimes v_{i+1} \otimes v_i \otimes \cdots \otimes v_d,\]
for any $v_1, \dots, v_d \in M$, with $v_i, v_{i+1}$ being $\Z/2$-homogeneous.  

We define the {\em $d$-th divided power} of $M$ to be 
the invariant subsupermodule 
\[\Gamma^d M:= (M^{\otimes d})^{\Si_d}\]
and set $\Gamma^0 M :=\k$.

\begin{remark}
We have used an equivalent definition to the  
one given in \cite{A1,A2}.  
See \cite[Example 2.1.1]{Krause}
for an analogous statement regarding 
divided powers of ordinary $\k$-modules.
\end{remark}

For any $M, M',N, N'\in \VV_\k$, there is an even isomorphism 
\begin{equation}\label{eq-isom}
 \Hom(M, M')\otimes \Hom(N, N')
 \simeq \Hom(M\otimes N, M'\otimes N')
\end{equation}
sending 
$\phi\otimes \psi \mapsto \phi\boxtimes \psi$, 
where 
\begin{equation*}
(\phi\boxtimes \psi)( v\otimes w):= 
(-1)^{\overline{\psi}\, \overline{v}}
 \phi(v) \otimes \psi(w).
\end{equation*}
Notice that the isomorphism \eqref{eq-isom} is natural with respect to composition.

Next, recall that the functor $\otimes^d: \VV_\k \to \VV_\k$ sending   
$M \mapsto M^{\otimes d}$,  whose action on morphisms is defined by 
\[ \otimes^d_{M,N}(\varphi) :=\,  
 \varphi \boxtimes \cdots \boxtimes \varphi: 
M^{\otimes d} \to N^{\otimes d} \] 
for any $\varphi\in \Hom(M,N)$.

Given $d,e\in \N$, there is an embedding 
\begin{equation}\label{eq-exp-1}
\Gamma^d M \otimes \Gamma^e N\, 
\hookrightarrow \, 
\Gamma^{d+e} (M\oplus N): \, 
x\otimes y \mapsto x\cdot y,
\end{equation}
where 
\[x\cdot y\, := \sum_{\sigma \in \Si_{d+e}/\Si_{d,e}} 
(x \otimes y). \sigma\, \in\, M^{\otimes (d+e)}\]
for $x\in M^{\otimes d}$, $y\in M^{\otimes e}$,
where $\Si_{d,e}$ denotes the image of $\Si_d \times \Si_e \hookrightarrow \Si_{d+e}$.
Then as in 
\cite[Eq.(26)]{A1}, 
there is a corresponding decomposition 
\begin{equation}\label{eq-exp-2}
\Gamma^d(M\oplus N) \cong 
\bigoplus_{0\leq c\leq d} \Gamma^cM\otimes \Gamma^{d-c}N 
\end{equation}
where we identify 
$\Gamma^cM\otimes \Gamma^{d-c}N$ 
as a subsupermodule of $\Gamma^d(M\oplus N)$ under the embedding \eqref{eq-exp-1}.

It follows easily from \eqref{eq-exp-2} 
that the divided power $\Gamma^d M$
of a projective $\k$-supermodule $M\in \VV_\k$ 
is again finitely-generated and projective. 
This yields a functor 
$\Gamma^d: \VV_\k \to \VV_\k$ 
which is a subfunctor of $\otimes^d$.  
In particular, the action of $\Gamma^d$ on morphisms 
is defined by restriction  
\begin{equation*}
\Gamma^d_{M,N}(\varphi) 
:= (\varphi^{\otimes d})|_{\Gamma^d M}:
\Gamma^dM \to \Gamma^d N
\end{equation*}
for any $\varphi\in \Hom(M,N)$. 
It is then clear that the (non-linear) maps $\Gamma^d_{M,N}$ 
are even for all $M,N\in \VV_\k$, 
so that $\Gamma^d$ restricts to a functor 
$\Gamma^d:\VV_\k^\ev \to \VV_\k^\ev$.

Suppose that 
$I = (d_1, \dots, d_s)$ is a tuple of positive 
integers, and let $\mathfrak{S}_I$ denote the subgroup 
$\mathfrak{S}_{d_1}\times \cdots \times \mathfrak{S}_{d_s} \subseteq \mathfrak{S}_{|I|}$, 
where $|I| := \sum d_i$. Given $M \in \VV_{\k}$ and 
distinct nonzero elements $v_1, \dots, v_s \in M_{\0}$, 
we define the new element
\[(v_1, \dots, v_s; I)_0 := 
\sum_{\sigma \in \mathfrak{S}_{|I|}/\mathfrak{S}_{I}} (v_1^{\otimes d_1} \otimes \cdots \otimes v_s^{\otimes d_s}).\sigma,\]
which belongs to 
$(M_{\0}^{\otimes |I|})^{\mathfrak{S}_{|I|}}$.  Similiarly, 
if $v_1', \dots, v_t' \in M_{\1}$, we define the (possibly 
zero) element
\[(v_1', \dots, v_t')_1:=
\sum_{\sigma \in \mathfrak{S}_{t}} (v_1' \otimes \cdots \otimes v_t').\sigma,\]
which belongs to $(M_{\1}^{\otimes t})^{\mathfrak{S}_t}$.

\begin{lemma}\label{lem:divided}
Suppose that $M\in \VV_\k$ is a free $\k$-supermodule, 
and that 
$\{v_1, \dots, v_\mu\}$ 
$($resp.~$\{v_{\mu+1}, \dots, v_{\mu+\nu}\})$ 
is an ordered basis of $M_{\0}$ $($resp.~$M_{\1})$.
Then a basis for $\Gamma^d M$ 
is given by the set of all elements of the form
\[(v_{i_1}, \dots, v_{i_s}; I)_0 \cdot  
(v_{j_1}, \dots, v_{j_t})_1,\]
such that 
$i_1, \dots, i_s$
$($resp.~$j_1,\dots, j_t)$ 
are distinct elements of the set
$\{1,\dots, \mu\}$
$($resp.~$\{\mu+1,\dots, \mu+\nu\})$
and $|I| +t = d$.
\end{lemma}

\begin{proof}
For any $\k$-module $V\in \V_\k$, let 
$D^d V$ and $\Lambda^d V$ denote the ordinary $d$-th divided 
and exterior powers of $V$, respectively. 
Then it follows, as in \cite[(8)]{A1}, 
that we have the following isomorphisms of $\k$-modules 
\begin{equation}\label{eq-Gamma}
\Gamma^d M\ \simeq\ \bigoplus_{k + l = d} 
D^k\abs{M_{\0}} \otimes \Lambda^l \abs{M_{\1}}
\ \simeq \bigoplus_{k+l =d} \Gamma^k M_{\0} \otimes \Gamma^l M_{\1}
\end{equation}
for each $d \ge 0$.  

One may check by comparison with Proposition 4 of
\cite[Ch.IV, \S5]{Bourbaki} that the set 
\[ \{ (v_{i_1}, \dots, v_{i_s}; I)_0 ;\ |I| = k \mbox{ and } 1\le i_1< \cdots< i_s \le \mu\}\]  
is a basis of $\Gamma^k M_{\0}$.  It is also 
not difficult to verify that
\[ \{ (v_{j_1}, \dots, v_{j_l})_1;\ 1 \le j_1< \cdots< j_l \le \nu\}\]  
is a basis of $\Gamma^l M_{\1}$.  The 
lemma then follows from (\ref{eq-Gamma}).
\end{proof}

\subsection{The category $\Gamma^d\C$}
Suppose $M, N\in \VV_\k$.  
Then define 
$\psi^d=\psi^d(M,N)$ to be 
the unique map which makes the  
following square commute: 
\begin{equation}\label{commute}
\begin{tikzcd}[row sep=large]
\Gamma^d M \otimes \Gamma^d N  
 \ar[d, tail ] \ar[r, "\psi^d"] 
& \Gamma^d(M\otimes N)
\ar[d, tail ]\\
M^{\otimes d} \otimes N^{\otimes d} 
\ar[r, "\sim"]
& (M \otimes N)^{\otimes d} 
\end{tikzcd}
\end{equation}

We are now ready to define divided powers of any given $\VV_\k^\ev$-enriched category. 

\begin{definition}
Suppose that $\C$ is a $\VV_\k^\ev$-enriched category. 
Then we define a new $\VV_\k^\ev$-category, $\Gamma^d\C$, which has the
same objects as $\C$, with morphisms given by
\[\Hom_{\Gamma^d\C}(X,Y):= \Gamma^d\Hom_{\C}(X,Y)\]
for all $X,Y \in \C$. 
The composition of morphisms 
\[\circ^{\Gamma^d\C}_{X,Y,Z}: \
\Hom_{\Gamma^d\C}(Y,Z) \otimes \Hom_{\Gamma^d(\C)}(X,Y)\, \rightarrow\, \Hom_{\Gamma^d\C}(X,Z)\]
is defined by the following composite
\[
\Gamma^d\Hom_{\C}(Y,Z) \otimes \Gamma^d\Hom_{\C}(X,Y) 
\ \xrightarrow{\,\psi^d\,}\  
\Gamma^d\big(\Hom_{\C}(Y,Z) \otimes \Hom_{\C}(X,Y) \big)
\]
\[\hspace{2.5cm} \xrightarrow{\,\Gamma^d(\circ^\C_{X,Y,Z})\,} \ 
\Gamma^d\Hom_{\V}(X,Z)\]
for all $X,Y,Z\in \C$.
\end{definition}

\subsection{The algebra $\Gamma^d A$}
Suppose that $A\in \VV_\k$ is a superalgebra.
Then we may identify  $A$ as a
$\VV_\k^\ev$-enriched category 
with one object, $o$, and hom-set  
$\Hom_A(o,o):= A$.
Hence, the divided power $\Gamma^d A$ of 
a superalgebra $A$ is also a superalgebra. 
The multiplication $m_{\Gamma^dA}$ is induced by the composition 
\[ \Gamma^d A\otimes \Gamma^d A
\xrightarrow{\psi^d} 
\Gamma^d(A\otimes A) 
\xrightarrow{\Gamma^d(m_A)}
\Gamma^d A. \] 
In particular, $\Gamma^d A$ is a subsuperalgebra 
of $A^{\otimes d}$.

\subsection{Generalized Schur algebras}
\label{ss-gen-schur}
Fix $m,n,d\in \N$, and let $A\in \VV_\k$ be a superalgebra.
Then we may identify $\mat_n(A)$ with the 
superalgebra of $n\times n$ matrices 
over $A$. 
We may also identify $\mat_n(A)$ with 
the tensor product $A\otimes \mat_n(\k)$ of superalgebras. 
The generalized Schur algebra 
$S^A(n,d)$ defined in \cite{EK1} may then be identified as the superalgebra 
\[S^A(n,d)
=\Gamma^d \mat_n(A).\]
The classical Schur algebra, $S(n,d)$, is  obtained as
the special case:
$S^\k(n,d) = \Gamma^d \mat_n(\k)$.
As mentioned in \cite[Example 2.1.1]{Krause},
this definition is equivalent to  
the original definition given by Green in \cite{Green}.

We next introduce another type of generalized Schur superalgebra
which will be needed later. 
Given nonnegative integers $m,n$, let 
\[S^A(m|n,d) := 
\Gamma^d \mat_{m|n}(A).\]
Notice that $S^\k(m|n,d)$ is isomorphic to 
the Schur superalgebra 
$S(m|n,d)$ introduced by Muir in \cite{Muir}.

\section{Generalized Polynomial Functors}\label{sec-gen-poly}
We now define a generalization of both the ordinary categories of strict polynomial functors, defined by Friedlander and Suslin in \cite{FS}, and  
the super-analogues of these categories defined in \cite{A1}.

\subsection{Representations of a category}\label{sec-rep}
Suppose that $\C$ is a $\VV_\k^\ev$-enriched category.
Then write 
\[\rrep\, \C := \mathrm{Fun}_\k(\C,\VV_\k).\]
to denote the category of all even $\k$-linear functors, 
$F:\C \to \VV_\k$.

Next suppose that $X\in\C$, and consider the superalgebra $E:= \End_\C(X)$. 
The categories $\rrep\, \C$ and $E\- \mod$ are both $\VV_\k^\ev$-enriched 
categories. Recall that 
$(E\-\mod)^\ev$ and $\VV_\k^\ev$ 
are both exact categories.
Now since direct sums, products, 
kernels and cokernels can be computed objectwise in
the target category $\VV_\k$, 
we see that $(\rrep\, \C)^\ev$ is also an exact category.

The relationship between $\rrep\, \C$ and $E\-\mod$ is given by evaluation 
on $X$.
If $F \in \rrep\, \C$, the (even) functoriality of $F$ makes the $\k$-supermodule 
$F(X)$ into a left $E$-supermodule.
We thus have the evaluation functor:
\[\rrep\, \C \to\  E\-\mod:\, F \mapsto F(X).\]

This evaluation functor also has the following interpretation.  
Since the covariant hom-functor 
$h^X:= \Hom_\C(X,-)$ is an even $\k$-linear functor, it must belong to 
$\rrep\, \C$.
In this situation, Yoneda's lemma takes the form of an even isomorphism
\begin{equation}\label{eq-Yoneda}
\Hom_{\rrep\, \C}(h^X, F)\ \simeq\ F(X)
\end{equation}
for any $F \in \rrep\,\C$.  In particular, 
\[ E\ =\ h^X(X)\ \simeq\ \End_{\rrep\,\C}(h^X).\]
Hence, Yoneda's lemma allows us to interpret ``evaluation at $X$" 
as the functor $\Hom_{\rrep\, \C}(h^X, -): \rrep\, \C \to{} E\-\mod$.

Notice that the parity change functor, $\Pi: \VV_\k \rightarrow \VV_\k$, induces
by composition a functor $\Pi\circ- : \rrep\, \C \to \rrep\, \C$.

We now describe a condition on $X$ which ensures that evaluation 
is in fact an equivalence of categories.  
The following lemma is a generalization of \cite[Proposition A.1]{Touze}.
It was stated in \cite[Proposition A.1]{A1} assuming that $\k$ is a field.  However, 
the generalization to the case of a commutative ring $\k$ is easily obtained. 

\begin{lemma}[{\cite{Touze,A1}}]
\label{lem:eval}
Let $\C$ be a $\VV_\k^\ev$-enriched category. Assume that there exists an object $P\in\C$ 
such that for all $X,Y\in\C$, the composition induces a surjective map
\[ \Hom_\C(P,Y)\otimes \Hom_\C(X,P)\twoheadrightarrow \Hom_\C(X,Y).\]
Then the following hold.
\begin{enumerate}
\item[(i)] For all $F\in  \rrep\, \C$ and all $Y\in\C$, the canonical map 
$\Hom_\C(P,Y)\otimes F(P)\to F(Y)$ is surjective.
\item[(ii)]  The set  $\{h^P, \Pi h^P\}$ is a projective generator of $(\rrep\, \C)^\ev$, 
where  $h^P = \Hom_\C(P, -)$ as above.
\item[(iii)] If $E = \End_\C(P)$, then evaluation on $P$ induces an 
equivalence of categories, $\rrep\, \C \simeq{}  E \- \mod$.
\end{enumerate}
\end{lemma}

We note that for a given a $\V_\k$-category $\C$, we may also consider the category 
\[\rep\, \C := \mathrm{Fun}_\k(\C,\V_\k)\]
consisting of ordinary 
 $\k$-linear functors and natural transformations.

\subsection{Generalized polynomial functors}
Given a $\VV_\k^\ev$-category $\C$
and $d\in\N_0$,
we associate
the corresponding category 
\[\PP_{d, \C} := \rrep\, \Gamma^d\C\]
which we call the category of
({\em homogeneous, degree $d$})
{\em generalized polynomial functors}
associated to the category $\C$.

Suppose now that $\C$ is a $\V_\k$-enriched category.
Then we may also consider the category
\[\P_{d,\C}:=\rep\, \Gamma^d\C\]
whose objects we again call generalized polynomial functors.

 \begin{example}\label{ex-pol}
Suppose that $\k$ is a field.
Then we have the following special cases:
\begin{itemize}
\item[(a)]
$\P_{d, \V_\k}$
is equivalent to the category 
$\P_{d,\k}$
of homogenous strict polynomial functors,
defined by Friedlander and Suslin
in \cite{FS}.
\smallskip
\item[(b)]
$\PP_{d, \VV_\k}$
is equivalent to the category
$\Pol^{\I}_d$
defined in \cite[Section 3]{A1}.
\smallskip
\item[(c)]
Recall from Remark \ref{rem-clifford}
 that $C(1)$ denotes the rank 1 Clifford algebra.
Then
$\PP_{d,\EE_{C(1)}}=\PP_{d,\VV_{C(1)}}$
is equivalent to the category 
$\Pol^{\II}_d$
of {\em spin polynomial functors} defined in \cite[Section 3]{A1}.
\end{itemize}
\end{example}

We are now ready to state our main result.

\begin{theorem}\label{thm-main-1}
Suppose $A\in \VV_\k$ is a superalgebra
which is free as a $\k$-module. 
Then for $m,n,d\in \N$, with $m, n\geq d$, 
 we have the following  equivalences of categories:
\smallskip
\begin{enumerate}
\item[(i)]
$\displaystyle
\PP_{d, \VV_{A}}\ \equi\ S^A(m|n,d)\-\mathbf{mod};$
\\
\item[(ii)]
$\displaystyle
\PP_{d, \EE_{A}}\ \equi\ 
S^A(n,d) \- \mathbf{mod}$,
\end{enumerate}
\smallskip
given by evaluation on $A^{m|n}$ and $A^n$, respectively.
\end{theorem}

\begin{proof}
We prove part (i). Since
$S^A(m|n,d)= \End_{\Gamma^d\VV_A}( A^{m|n})$, 
it suffices by Lemma \ref{lem:eval}
to show that the composition map
 \begin{equation}\label{eq-comp}
 \Hom_{\Gamma^d\VV_A}(A^{m|n}, Y) \otimes \Hom_{\Gamma^d\VV_A}(X, A^{m|n}) 
 \rightarrow \Hom_{\Gamma^d\VV_A}(X,Y)
 \end{equation}
is surjective, for all $X,Y\in \VV_A$. 
 
By naturality, one may reduce to the case where $X$ is
a free (right) $A$-supermodule.  
Since $Y$ is projective and thus a direct summand 
of a free  $A$-supermodule, 
we may also assume that $Y$ is free.
We may thus assume that 
 $X= A^{p|q}$, $Y=A^{s|t}$, for some $p,q,s,t\in \N_0$.

Let us identify 
$\Hom(\k^{p|q}, \k^{m|n})$ as a subset of 
$\Hom(X, A^{m|n})$
via the embedding of 
Lemma \ref{lem:free}.(i).
It will then suffice to show the following map
\begin{equation}\label{eq-comp2}
\Gamma^d\Hom_A(A^{m|n}, Y) \otimes \Gamma^d \Hom_\k(\k^{p|q}, \k^{m|n})
\rightarrow \Gamma^d \Hom_A(X,Y),
\end{equation}
obtained by restricting composition, is surjective.

Suppose that 
$a_1, \dots, a_{\mu}$ 
(resp.~$a_{\mu+1}, \dots, a_{\mu+\nu}$)
is a $\k$-basis of $A_{\0}$ (resp.~$A_{\1}$). 
Then, using the notation in the proof of 
Lemma \ref{lem:free},
it follows that
there exist (homogeneous) 
bases 
$(e(j,i))$, $(e^{(r)}(k,j))$ and 
$(e^{(r)}(k,i))$ 
of $\Hom(\k^{p|q}, \k^{m|n})$, 
$\Hom_{A}(A^{m|n}, Y)$ 
and $\Hom_{A}(X, Y)$, respectively, such that:
%:
\[  e^{(r)}(k,j_1) \circ e(j_2,i)\, =\ 
e^{(r)}(k,i)\cdot\delta_{j_1,j_2},\]
where $\delta_{j_1,j_2}$ is the Kronecker delta.

Suppose that 
$r_1, \dots, r_a, r_1', \dots, r_b' 
\in \{1, \dots, \mu+\nu\}$, 
 for some $1\le a, b \le d$,
 and assume the basis elements
\[e^{(r_1)}(k_1,i_1), \dots, e^{(r_a)}(k_a,i_a)\quad
(\text{resp.~}e^{(r_1')}(k_1',i_1'), \dots, e^{(r_b')}(k_b',i_b'))\]
are even (odd).
 Then to prove surjectivity, it suffices to show 
 that each elemen of the form 
\begin{equation}\label{eq-element}
 (e^{(r_1)}(k_1,i_1), \dots, e^{(r_a)}(k_a,i_a); I)_0\cdot 
 (e^{(r_1')}(k_1',i_1'), \dots, e^{(r_b')}(k_b',i_b'))_1
\end{equation}
belongs to the image of (\ref{eq-comp2}), since $\Gamma^d \Hom_{A}(X,Y)$ is 
spanned by such elements according to Lemma \ref{lem:divided}.

Now suppose  $I=(d_1, \dots, d_a) \in (\Z_{>0})^a$.
Then, since $m, n\ge d= |I| + b \ge a+b$, we may choose 
{\em distinct} indices $j_1, \dots, j_a, j_1', \dots, j_b'$ 
such that 
$e(j_1,i_1), \dots, e(j_a,i_a),\ e(j_1',i_1'), \dots, e(j_b',i_b')$
are all even elements.
We are thus able to form the element
\begin{align*}
 &(e^{(r_1)}(k_1,j_1), \dots, e^{(r_a)}(k_a,j_a); I)_0 \cdot
(e^{(r_1')}(k_1',j_1'), \dots, e^{(r_b')}(k_b',j_b'))_1\\[1pt]
 &\quad \otimes (e(j_1,i_1), \dots, e(j_s,i_s),\ e(j_1',i_1'), \dots, e(j_t',i_t'); J)_0,
\end{align*}
where $J = (d_1, \dots, d_a, 1,\dots, 1) \in (\Z_{>0})^{a+b}$, which is sent to 
the element (\ref{eq-element}) under the map induced by composition 
in $\Gamma^d\VV_A$.

The proof of equivalence (ii) is similar. 
\end{proof}

We also have the following non-graded analogue of Theorem \ref{thm-main-1}.

\begin{theorem}\label{thm-main-2}
Let $A\in \V_\k$ be an ordinary $\k$-algebra which is free as a $\k$-module,
and suppose $n,d \in \mathbb{N}$.
Then there is an equivalence of categories:
\[\P_{d, \V_{A}}
\ \equi\ 
S^A(n,d) \- \mathrm{mod},\]
provided $n \ge d$.
\end{theorem}

\begin{proof}
The proof is analagous to Theorem \ref{thm-main-1}, using
\cite[Prop.~A.1]{Touze}
in place of Lemma \ref{lem:eval}.
\end{proof}

It follows by Section \ref{ss-gen-schur} and
 Example \ref{ex-pol}
that Theorem \ref{thm-main-1}
generalizes Theorem 4.2 of \cite{A1}, while
Theorem \ref{thm-main-2} generalizes Theorem 3.2 of \cite{FS}.

\section{Generalized Schur-Weyl Duality}\label{sec-gen-sw} 
Throughout this section, we fix a superalgebra $A\in \VV_\k$.

\subsection{Wreath products}
Consider the group algebra of the symmetric group,
 $\k \Si_d$,
 as a superalgebra concentrated in degree zero.
Then the {\em wreath product superalgebra}, $A\wr \Si_d$, 
is the $\k$-supermodule $A^{\otimes d} \otimes \k \Si_d$, 
with multiplication defined by 
\begin{equation}\label{wr}
(x \otimes \rho ) \cdot (y \otimes \sigma)
:= x (y\rho^{-1}) \otimes \rho \sigma
\end{equation}
for all $x,y \in A^{\otimes d}$ and $\rho, \sigma \in \Si_d$. 
If $G$ is a finite group, then note for example that 
$(\k G) \wr \Si_d$ is isomorphic to 
the group algebra of the classical wreath product, 
$G\wr \Si_d := G^{d} \rtimes \Si_d$. 

Assume in the remainder that $A$ is free  
as a $\k$-module. 
We then identify the tensor power $A^{\otimes d}$ 
and group algebra  $\k\Si_d$  
as subalgebras of $A\wr \Si_d$ by setting 
\[A^{\otimes d}=A^{\otimes d} \otimes 1_{\Si_d}, \quad 
\k\Si_d=1_{A^{\otimes d}} \otimes \k \Si_d\] 
respectively.

\subsection{Tensor products}
Given nonnegative integers $d$ and $e$, we have an embedding 
$\Si_d \times \Si_e \hookrightarrow \Si_{d+e}$.  This induces an embedding 
\begin{equation}\label{eq-embed}
\Gamma^{d+e}M \hookrightarrow \Gamma^{d}M\otimes \Gamma^{e}M
\end{equation}
 for any $M\in \VV_\k$, given by the composition of the following maps:
\begin{align*}
 \Gamma^{d+e}M =
(M^{\otimes d+e})^{\Si_{d+e}}\ &\hookrightarrow\,
 (M^{\otimes d+e})^{\Si_d\times \Si_e}\\
 &\, \simeq\ \,
  (M^{\otimes d})^{\Si_d}\otimes (M^{\otimes e})^{\Si_e} = 
  \Gamma^dM \otimes \Gamma^eM.
\end{align*}
We then have a corresponding diagonal embedding 
\begin{equation}\label{eq-embed-2}
\delta: \Gamma^{d+e}\,\C \hookrightarrow \Gamma^d\C \otimes \Gamma^e\C, \qquad  
\end{equation}
which sends $X \mapsto (X,X)$
and which acts on morphisms via the 
embedding induced by (\ref{eq-embed}).

Given $F\in \PP_{d,\C}$, $G\in \PP_{e,\C}$, let
$F\boxtimes G \in \rrep(\Gamma^d\C\otimes\Gamma^e\C)$
denote the functor which sends 
\[(X_i,Y_i) \mapsto F(X_i)\otimes G(Y_i)
\ \mbox{ and } \ 
\varphi_1\otimes \varphi_2 \mapsto F(\varphi_1)\boxtimes G(\varphi_2),\]
%:
for all $X_i,Y_i \in \C$ and 
$\varphi_i\in\Hom_\C(X_i,Y_i)$, for $i=1,2$, respectively.
We then define the tensor product
\begin{equation}\label{eq-tensor}
-\otimes-: \PP_{d, \C} \times \PP_{e, \C} \rightarrow \PP_{d+e,\, \C}
\end{equation}
as the composition of functors
\[
\PP_{d,\C} \times \PP_{e,\C}
\xrightarrow{\ -\boxtimes -\ }
\rrep(
\Gamma^d\C
\otimes
\Gamma^e\C)
\xrightarrow{\  \delta_* \ }
\PP_{d+e,\C}.
\]
where $\delta_*$ is the functor induced by \eqref{eq-embed-2}.

\subsection{Generalized Schur-Weyl duality}
Let us write 
$\PP^A_d:= \PP_{d,\EE_A}$.
Then 
the (external) tensor product defined above
induces a bifunctor
\[ -\otimes- : 
\PP^A_d \times \PP^A_e
\rightarrow
\PP^A_{d+e}.\]
This gives
$\PP^A:= \bigoplus_{d\ge 0} \PP^A_d$
the structure of a monoidal category.
\smallskip

Given any object $X\in \EE_A$,
we associate the corresponding object
\[\Gamma^{d,X} := \Hom_{\Gamma^d\EE_A}(X,-)\]
in the category $\PP^A_d$.
It follows by Yoneda's lemma \eqref{eq-Yoneda}
that $\Gamma^{d,X}$ is a projective object.

Let us write
$\Gamma_A^{d,n}:= \Gamma^{d,A^n}$.
Then it follows from 
(\ref{eq-exp-1}) that we have a decomposition 
\begin{equation}\label{eq-decomp.}
\Gamma_A^{d,m+n} \simeq \bigoplus_{i+j=d} \Gamma_A^{i, m} \otimes \Gamma_A^{j, n}
\end{equation}
of strict polynomial functors.

Now let $\Lambda(n,d)$ denote the set of all tuples 
$\lambda=(\lambda_1, \dots, \lambda_n)\in (\Z_{\geq 0})^n$ 
such that $\sum \lambda_i = d$.  Given $\lambda \in \Lambda(n,d)$, 
we will write 
$\Gamma^\lambda_A :=
\Gamma^{\lambda_1,1}_A \otimes \cdots \otimes \Gamma^{\lambda_n,1}_A$.
By (\ref{eq-decomp.}) and induction, we have a canonical isomorphism
\begin{equation}\label{eq-lambda}
\Gamma_A^{d,n} \simeq \bigoplus_{\lambda\in \Lambda(n,d)} \Gamma_A^\lambda.
\end{equation}
It follows that the objects $\Gamma^\lambda$ are projective in 
$\PP^A_d$.

Let $\omega=(1,\dots, 1) \in \Lambda(d,d)$.  
Notice that 
$\Gamma^\omega(X) \simeq X^{\otimes d}$
for any $X\in \EE_A$,
since $\Hom_A(A,X)\simeq X$.
It follows that
 $\otimes^d:=\Gamma^\omega$
is  a projective object of $\PP^A_d$.
We then have the following 
analogue of the generalized Schur Weyl duality 
described in \cite{EK1}.

\begin{theorem}\label{thm-gen-sw}
Assume $n \ge d$, and write $V=A^n$.
\begin{itemize}
\item[(i)] The left $S^A(n,d)$-supermodule 
%:
$V^{\otimes d} \simeq \Hom_{\PP^A_d}
(\Gamma^{d,n}, \otimes^d)$ 
is a projective object of $S^A(n,d) \-\mod$.
\smallskip
\item[(ii)]
There is a canonical isomorphism of superalgebras: 
\[\End_{\PP^A_d}(\otimes^d) \cong A\wr\Si_d.\]

\item[(iii)] We have an exact functor 
%:
\[\Hom_{\PP^A_d}(\otimes^d, -): \PP^A_d 
\rightarrow A\wr\Si_d \,\- \mod.\]
\end{itemize}
\end{theorem}

\begin{proof}
Part (i) follows from Theorem \ref{thm-main-1}.(i) and the fact that
 $\bigotimes^d$ is a projective object of $\PP^A_d$.

Recall from \cite[\S 5]{EK1} that there is an even idempotent
$\xi_\omega \in S^A(n,d)$
such that $V^{\otimes d} \simeq S^A(n,d) \xi_\omega$,
as left supermodules.
Since $\xi_\omega$ is an idempotent, it follows that the superalgebras
$\xi_\omega S^A(n,d) \xi_\omega$ and 
$\End_{S^A(n,d)}(S^A(n,d) \xi_\omega)$ are isomorphic.
Part (ii) thus follows from Theorem \ref{thm-main-1}.(i), 
together with 
\cite[Lemma 5.15]{EK1} and \cite[Proposition 5.17]{EK1}.

Finally, (iii) is a direct consequence of (i) and (ii).
\end{proof}

\begin{remark}
One may refer to the functor in Theorem \ref{thm-gen-sw}.(iii) 
as the {\em generalized Schur-Weyl duality functor}.  A similar functor 
related to classical Schur-Weyl duality was studied in 
\cite{HY} in the context of $\mathfrak{g}$-categorification.
\end{remark}

\end{document}